\documentclass[11pt]{amsart}
\usepackage{amssymb,enumerate}
\usepackage{amsmath}
\usepackage{bbm}           
\usepackage{bm}
\usepackage{eso-pic,graphicx}
\usepackage{tikz}

\usepackage[colorlinks=true, pdfstartview=FitV, linkcolor=blue, citecolor=blue, urlcolor=blue]{hyperref}

\makeatletter
\def\Ddots{\mathinner{\mkern1mu\raise\p@
\vbox{\kern7\p@\hbox{.}}\mkern2mu
\raise4\p@\hbox{.}\mkern2mu\raise7\p@\hbox{.}\mkern1mu}}
\makeatother

\def\Xint#1{\mathchoice
{\XXint\displaystyle\textstyle{#1}}%
{\XXint\textstyle\scriptstyle{#1}}%
{\XXint\scriptstyle\scriptscriptstyle{#1}}%
{\XXint\scriptscriptstyle\scriptscriptstyle{#1}}%
\!\int}
\def\XXint#1#2#3{{\setbox0=\hbox{$#1{#2#3}{\int}$}
\vcenter{\hbox{$#2#3$}}\kern-.5\wd0}}

\def\dashint{\Xint-}

\newtheorem{theorem}{Theorem}[section]

\theoremstyle{definition}

\newtheorem{lemma}[theorem]{Lemma}

\def\al{{\alpha}}
\def\R{\mathbb R}

\def\ra{\rightarrow}
\def\bey{\begin{eqnarray*}}
\def\eey{\end{eqnarray*}}

\DeclareMathOperator{\supp}{supp}

\def\M{{\mathcal M}}

\newcommand{\cz}{\text{Calder\'on-Zygmund}}

\newcommand{\eps}{\epsilon}
\def\({\left(}
\def\){\right)}
\def\[{\left[}
\def\]{\right]}
\def\<{\langle}
\def\>{\rangle}

\newcommand{\nR}{{\mathbb R}}

\begin{document}


\title[Characterization of compactness of bilinear commutators]{Characterization of compactness of the commutators of bilinear fractional integral operators}

\date{\today}

\author[L. Chaffee]{Lucas Chaffee}
\address{%
Department of Mathematics\\
University of Kansas\\
Lawrence, KS 66045, USA}
\email{chaffel@ku.edu}

\author[R. H. Torres]{Rodolfo H. Torres}
\address{%
Department of Mathematics\\
University of Kansas\\
Lawrence, KS 66045, USA}
\email{torres@ku.edu}

\thanks{Both authors partially supported by NSF grant DMS 1069015.}

\begin{abstract}
{The compactness of the commutators of bilinear fractional integral operators and point-wise multiplication, acting on products of Lebesgue spaces, is characterized in terms of appropriate mean oscillation properties of their symbols. The compactness of the commutators when acting on product of weighted Lebesgue spaces is also studied.}
\end{abstract}

\maketitle

\section{Introduction}

The purpose of this article is first: to characterize the compactness of the commutators of bilinear fractional integral operators with pointwise multiplication acting on product of Lebesgue spaces; and second: to obtain conditions on multiple weights, which yield compactness on the weighted Lebesgue spaces
(precise definitions are given in the next section).

We briefly summarize some classical and recent works in the literature, which lead to the results presented here. The first result on compactness of commutators of singular integrals with point-wise multiplication is due to Uchiyama \cite{U}. He refined the boundedness results of Coifman, Rochberg and Weiss \cite{CRW} on the commutator with symbols in the John-Nirenberg space $BMO$ to compactness. This is achieved by requiring the symbol to be not just in $BMO$, but rather in $CMO$, which is the closure in $BMO$ of the space of $C^\infty$ functions with compact support. For linear fractional integrals, the characterization of boundedness of the commutator was established by Chanillo \cite{Ch}, while the one for compactness is credited in Chen, Ding and Wang \cite{CDW} to Wang \cite{W}. As in the case of singular integrals of \cz\, type, the conditions are again that the symbol is respectively in $BMO$ or $CMO$.

In the multilinear setting, commutators of \cz \, operators and fractional integrals started to receive attention only a few years ago. For the \cz \, operators, as defined by Grafakos and Torres \cite{GT},  the main boundedness results for commutators with symbols in $BMO$ were obtained by P\'erez and Torres \cite{PT}, Tang \cite{tang}, Lerner et al. \cite{LOPTT} and P\'erez et al \cite{PPTT}. Meanwhile, for the commutator of multilinear fractional integrals, one can cite
the works of Lian and Wu \cite{LW}, Chen and Xue \cite{CX} and Chen and Wu \cite{CW}. Some of these works include weighted estimates as well. Compactness results in the multilinear setting have just began to be studied. In particular, B\'enyi and Torres \cite{BT} and B\'enyi et al. \cite{BDMT} showed that symbols in $CMO$  again produce compact commutators.  Very recently, Chaffee \cite{Chaff} proved that the symbols must be in $BMO$ to obtain boundedness of the commutators. Here we will show that the smaller space $CMO$ in fact characterizes compactness in the bilinear setting. In the process we will also obtain a result about compactness of commutators with bilinear fractional integrals on weighted Lebesgue spaces. This last result complements the results of B\'enyi et al. \cite{BDMT2} for bilinear \cz \, operators and it is of interest in its own. Formally, the characterization results for $\al=0$ would correspond to the case of bilinear \cz \, operators. However, some of the techniques employed in this paper do not apply to \cz \, operators, mainly because they lack positive kernels. We intend to study the case of \cz \,  operators in future work.

\section{Definitions and preliminaries}

As usual,  $BMO$ is the space of all locally integrable functions $b$ such that
$$\|b\|_{BMO}:=\sup_{Q}\dashint_Q |b(x)-\dashint_Q b|\, dx < \infty,$$
where the  supremum is taken over all cubes $Q\in \R^n$  with sides parallel to the coordinate axes, and  $\dashint_Q b= b_Q$ is the average of $b$ over $Q$.
Also, as mentioned in the introduction,  $CMO$ is the closure in the $BMO$ norm of $C^\infty_c(\R^n)$, which represents
the space of infinitely differentiable functions with compact support. It was shown in \cite{U} that $CMO$ can be characterized in the following way.

{\it A function $b \in BMO$ is in $CMO$ if an only if,

\begin{align}
  & \displaystyle\lim_{a\to0}\sup_{|Q|=a}\frac{1}{|Q|}\int_Q|b(x)-b_Q|dx=0, \label{1cmo}\\
   & \displaystyle\lim_{a\to\infty}\sup_{|Q|=a}\frac{1}{|Q|}\int_Q|b(x)-b_Q|dx=0, \label{2cmo}\\
    & \displaystyle\lim_{|y|\to\infty}\frac{1}{|Q|}\int_Q|b(x+y)-b_Q|dx=0, \mbox{ for each } Q.\label{3cmo}
    \end{align}
}

For $0<\alpha<2n$ the bilinear fractional integral operator $\mathcal I_\alpha$ is a priori defined for $f,g\in C^\infty_c$ by
\begin{equation}\label{Ialphacal}
\mathcal I_\alpha(f,g)(x):=\iint_{\R^{2n}}\frac{1}{(|x-y|+|x-z|)^{2n-\alpha}}\,f(y)g(z)\,dydz\nonumber.
\end{equation}
For convenience we will consider here the equivalent operator
\begin{equation}\label{Ialpha}
I_\alpha(f,g)(x):=\iint_{\R^{2n}}\frac{1}{(|x-y|^2+|x-z|^2)^{n-\alpha/2}}\,f(y)g(z)\,dydz \nonumber.
\end{equation}
Its commutators with {\it symbol} $b\in BMO$  are given by
$$[b,I_\alpha]_1(f,g):=I_\alpha(bf,g)-bI_\alpha(f,g)$$
and
$$[b,I_\alpha]_2(f,g):=I_\alpha(f,bg)-bI_\alpha(f,g).$$
By symmetry, it would be enough in what follows to consider one of these two commutators, say $[b,I_\alpha]_1$, and we will do so.

For $1<p<\infty$, recall that the Muckenhoupt class $A_p$ of weights consists of all non-negative, locally integrable, functions $w$ such that
\begin{equation}\label{Ap}
[w]_{A_p}:=\sup_Q \left(\dashint_Q w\right)\left(\dashint_Q w^{1-p'}\right)^{\frac{p}{p'}}<\infty;\nonumber
\end{equation}
while    $A_{\infty} = \cup_{1<p<\infty} A_p$. For $1<p\leq q<\infty$, the weight $w$ is in $A_{p,q}$ if
$$[w]_{A_{p,q}}:=\sup_{Q}\left(\dashint_Q w^{q}\right) \left( \dashint_Q w^{-p'}\right)^{q/p'}<\infty.$$
It is easy to see that
$$
[w]_{A_{p,q}}=[w^q]_{A_{1+q/p'}}.
$$

We also recall the definition of the multiple or vector weights used in the bilinear setting. For $1<p_1,p_2<\infty$, ${\bf P}=(p_1,p_2)$, and $p$ such that $\frac1p=\frac1{p_1}+\frac1{p_2}$,
a vector weight ${\bf w}=(w_1,w_2)$ belongs to ${\bf A}_{\bf P}$ if
$$[{\bf w}]_{{\bf A}_{\bf P}}:=\sup_{Q}\left(\dashint_Q w_1^{p/p_1}w_2^{p/p_2}\right)
\left(\dashint_Q w_1^{1-p_1'}\right)^{p/p_1'} \left(\dashint_Q w_2^{1-p_2'}\right)^{p/p_2'}<\infty.$$
For brevity, we will often use the notation $\nu_{\bf w}=w_1^{p/p_1}w_2^{p/p_2}$ in the first integral.
We note that in \cite{LOPTT} it was shown that for ${\bf w} \in {\bf A}_{\bf P}$, it holds that
$\nu_{\bf w}\in A_{2p}$, and that
\begin{equation}\label{holderconsequence}
A_{p_1}\times A_{p_2}\subsetneq {\bf A}_{\bf P}\subsetneq {\bf A}_{c\bf P}, \nonumber
\end{equation}
 for $c>1$.

For $1<p_1,p_2<\infty$, ${\bf P}=(p_1,p_2)$,  $0<\alpha< 2n$, $\frac\alpha{n} < \frac1{p_1}+\frac1{p_2}$,
and $q$ such that $\frac1q=\frac1{p_1}+\frac1{p_2}-\frac\alpha n$,
a vector weight ${\bf w}=(w_1,w_2)$ belongs to ${\bf A}_{{\bf P},q}$ if
$$[{\bf w}]_{{\bf A}_{{\bf P},q}}:=\sup_{Q}\left(\dashint_Q w_1^{q}w_2^{q}\right)  \left(\dashint_Q w_1^{-p_1'}\right)^{q/p_1'}
\left(\dashint_Q w_2^{-p_2'}\right)^{q/p_2'}<\infty .
$$
As with the $A_{\bf P}$ weights, for brevity we will use $\mu_{\bf w}=w_1^{q}w_2^{q}$. To avoid ambiguities in the notation we will use $\nu_{\bf w}$ when dealing with  ${\bf A}_{\bf P}$  classes and $\mu_{\bf w}$ with ${\bf A}_{{\bf P},q}$ ones.
It was shown by Moen in \cite{M} that if   ${\bf w} \in {\bf A}_{{\bf P},q}$  then  $w_i^{-p_i'}\in A_{2p_i'}$ and $\mu_{\bf w} \in A_{2q}$.
In addition, the weights in  ${\bf A}_{{\bf P},q}$  are precisely those for which
$$I_\alpha: L^{p_1}(w_1^{p_1})\times L^{p_2}(w_2^{p_2})\to L^q(\mu_{\bf w})$$
 is bounded.

A useful tool when studying bilinear fractional singular integrals is the corresponding maximal function
$$\M_\al(f,g)(x)=\sup_{Q\ni x} |Q|^{\al/n}\Big(\dashint_Q|f(y)|\,dy\Big)\Big(\dashint_Q|g(z)|\,dz\Big),$$
which also satisfies the bounds
$$\M_\alpha: L^{p_1}(w_1^{p_1})\times L^{p_2}(w_2^{p_2})\to L^q(\mu_{\bf w})$$
for the same parameter as $I_\alpha$. See \cite{M}.

The classes  ${\bf A}_{{\bf P},q}$ are also the natural ones for the the boundedness of commutators of bilinear fractional integral operators. In fact, it was first shown in \cite{CX} that given $0<\alpha<2n$, $1<p_1,p_2<\infty$, $1/p=1/p_1 + 1/p_2$  and $1/q=1/p - \alpha/n$, if
$(w_1^r, w_2^r) \in  {\bf A}_{{\bf P/r},q/r}$ for some $r>1$ with  $0<r\alpha<2n$, and $\mu_{\bf w} \in A_\infty$, then
$$[b,I_\alpha]_j: L^{p_1}(w_1^{p_1})\times L^{p_2}(w_2^{p_2})\to L^q(\mu_{\bf w}). $$
Moreover, the operator norm satisfies
\begin{equation}
\|[b,I_\alpha]_j\| \lesssim  \|b\|_{BMO}.\label{operatornorm}
\end{equation}
 Later on, in \cite{CW}, the result was improved and the explicitly stated bump condition involving $r>1$ was removed. This requires a simple argument based on reverse H\"older inequality, as used in the work \cite{LOPTT} when dealing with similar situation for the classes ${\bf A}_{\bf P}$. In fact, such condition is always satisfied: for $(w_1, w_2)  \in {\bf A}_{{\bf P},q}$ there exist an appropriate $r>1$, depending on
$(w_1, w_2)$, such that
$(w_1^r, w_2^r)  \in {\bf A}_{{\bf P/r},q/r}$; while it is also true that   $(w_1^r, w_2^r)  \in {\bf A}_{{\bf P/r},q/r}$ always implies $(w_1, w_2)  \in {\bf A}_{{\bf P},q}$ for all $r>1$.  \\

We now show two important properties of the weights we will be using, which in particular guarantee the boundedness of the commutators.

\begin{lemma} \label{conditions-weights}
{\it Let $1<p_1,p_2<\infty$, ${\bf P}=(p_1,p_2)$,  $0<\alpha< 2n$, $\frac\alpha{n} < \frac1{p_1}+\frac1{p_2}$,
and $q$ such that $\frac1q=\frac1{p_1}+\frac1{p_2}-\frac\alpha n$. Suppose that
$w_1^{\frac{p_1q}p},w_2^{\frac{p_2q}{p}}\in A_p$. Then, 
\begin{enumerate}[{\rm (i)}]
\item ${\bf w}=(w_1,w_2)\in {\bf A}_{{\bf P},q}$,
\item $\mu_{\bf w}=w_1^{q}w_2^{q}\in A_p\subset A_q$.
\end{enumerate}
}
\end{lemma}

\begin{proof}
Note that $p<\min\{p_1,p_2\}$, so $(w_1^{\frac{p_1q}p},w_2^{\frac{p_2q}{p}})\in {\bf A}_{\bf P}$, and we have
    \begin{align*}
    \left(\dashint_Q \left( w_1w_2 \right)^q \right) & \prod_{i=1}^2\left( \dashint_Q  w_i^{-p_i'} \right)^{q/p_i'}\\
    & = \left( \dashint_Q  \left( w_1^{p_1q/p} \right)^{p/p_1} \left( w_2^{p_2q/p}\right)^{p/p_2} \right)
    \prod_{i=1}^2\left( \dashint_Q w_i^{-p_i'} \right)^{q/p_i'}\\
    & \leq \left(\dashint_Q \left( w_1^{p_1q/p} \right)^{p/p_1} \left( w_2^{p_2q/p} \right)^{p/p_2}\right)
    \prod_{i=1}^2\left( \dashint_Q  \left(w_i^{p_iq/p}\right)^{1-p_i'}\right)^{p/p_i'} 
    \\
    & = \left[\left( w_1^{p_1q/p},w_2^{p_2q/p} \right)\right]_{{\bf A}_{\bf P}} <\infty.
    \end{align*}
A quick application of H\"older to the $A_p$ condition shows that
$$[\mu_{\bf w}]_{A_p} =  \sup_{Q}\left(\dashint_Q w_1^{q}w_2^{q}\right)
\left(\dashint_Q (w_1w_2)^{q(1-p')}\right)^{p/p'} $$
$$\leq \left[ w_1^{\frac{p_1q}p} \right]_{A_{p}}^{\frac{p}{p_1}}
\left[ w_2^{\frac{p_2q}p} \right]_{A_p}^{\frac p{p_2}} < \infty, $$
and since $q>p$, we also have $w_1^qw_2^q\in A_q$.\\
\end{proof}

As in other works in the literature dealing with compactness of singular integrals (see \cite{BDMT} and the references therein), we find it convenient to
use smooth truncations of $I_\alpha$.  Following the construction in \cite{BDMT} it is possible to approximate $I_\alpha$ by operators $I_\alpha^\delta$
defined by a smooth kernel  $K^\delta(x,y,z)$ in $\R^{3n}$ such that
$$
K^\delta(x,y,z) = \frac{1}{(|x-y|^2+|x-z|^2)^{n-\alpha/2}}
$$
for $\max(|x-y|,|x-z|) > 2\delta$;
$$
K^\delta(x,y,z) = 0
$$
for $\max(|x-y|,|x-z|) < \delta$;
and
$$
|\partial^\gamma K^\delta(x,y,z)| \lesssim \frac{1} {(|x-y|+|x-z|)^{2n-\alpha-| \gamma |}}
$$
for all $(x,y,z)$ and all multi-indexes with $|\gamma|\leq 1$.\\

The operators $I_\alpha^\delta$ approximate $I_\alpha$ in the following sense.

\begin{lemma}\label{approximation}
If $b\in C_c^\infty$ and  ${\bf w}\in {\bf A}_{{\bf P},q}$, then
$$\lim_{\delta \to 0} \| [b,I_\alpha^\delta] -[b, I_\al]\|_{L^{p_1}(w_1^{p_1})\times L^{p_2}(w_2^{p_2}) \to L^p(\mu_{\bf w})}=0.$$
\end{lemma}

\smallskip

The proof of this result is very similar to that of Lemma 2.1 in \cite{BDMT} and it is left to the reader.\\

We use the following definition of compactness in the bilinear setting. A bilinear operator is
{\it compact}  from $L^{p_1}(w_1) \times  L^{p_2}(w_2) \to L^{p_3}(w_3)$, if it maps the set
$$\{(f,g):\|f\|_{L^{p_1}(w_1)}\leq 1,\|g\|_{L^{p_2}(w_2)}\leq 1\}$$
into a pre-compact set  in $L^{p_3}(w_3)$. See \cite{BT} for natural properties of compact bilinear operator.\\

A criteria for compactness in weighted $L^q$ spaces is provided by the following weighted version of the Frech\'et-Kolmogorov-Riesz theorem.
We refer to the works by  Hanche-Olsen and Holden\cite{HOH} and  Clop and Cruz \cite{CC}.\\

\emph{Let $1<q<\infty$ and $w\in A_q$ and let $\mathcal K\subset L^q(w)$. If
\begin{align}
&\mathcal K \mbox{ is bounded in }L^q(w); \label{1compact}\\
&\displaystyle\lim_{A\ra \infty} \int_{|x|>A}|f(x)|^q\,w(x) \,dx=0 \mbox{ uniformly for }f\in \mathcal K; \label{2compact}\\
&\displaystyle\lim_{t\ra 0}\|f(\cdot+t)-f\|_{L^q(w)}=0 \mbox{ uniformly for }f\in \mathcal K; \label{3compact}
\end{align}
then $\mathcal K$ is pre-compact in $L^q(w)$.\\
}

A compact operator is bounded, so by the results in \cite{Chaff} the symbol of a compact operator must be at least in $BMO$.
It was also proved in \cite{BDMT} that $[b,I_\alpha]_j(f,g)$ is compact when the symbol $b$ is in $CMO$. The result we will show
establishes the necessity of this condition as well as the compactness of the commutators on appropriate weighted spaces.

\section{Main result}

\begin{theorem} \label{tfae}
Let $1<p_1,p_2<\infty$, ${\bf P}=(p_1,p_2)$,  $\frac1{p}=\frac1{p_1}+\frac1{p_2}$, $0<\alpha< 2n$, $\frac\alpha{n} < \frac1{p_1}+\frac1{p_2}$,
and $q$ such that $\frac1q=\frac1{p_1}+\frac1{p_2}-\frac\alpha n$ and $1<p,q<\infty$. Then the following are equivalent,
\begin{enumerate}[{\rm (i)}]
\item $b\in CMO$.
\item $[b,I_\alpha]_1:L^{p_1}(w^{p_1})\times L^{p_2}(w^{p_2})\to L^q(w_1^qw_2^q)$ is a compact operator for all ${\bf w} =(w_1, w_2)$
such that $w_1^{\frac{p_1q}p},w_2^{\frac{p_2q}{p}}\in A_p$.
\item $[b,I_\alpha]_1:L^{p_1}\times L^{p_2}\to L^q$ is a compact operator.
\end{enumerate}
\end{theorem}

\begin{proof} To prove that (i) implies (ii),  it is enough to assume that $b\in C^\infty_c$, and show that the image of $B_1(L^{p_1}(w_1^{p_1}))\times B_1(L^{p_2}(w_2^{p_2}))$ under $[b,I_\alpha^\delta]_1$ verifies  the Frech\'et-Kolmogorov-Riesz conditions in $L^q(\mu_{\bf w})$.\footnote{This follows from Lemma \ref{approximation}, the norm estimate \eqref{operatornorm}, and basic properties of compact operators.} The approach for this part is similar to that in \cite{BDMT} but we need to carefully use the properties of the weights established in Lemma \ref{conditions-weights}.

Note that \eqref{1compact} is immediate since for $b\in C^\infty_c$, $[b,I_\alpha^\delta]_1$ is bounded  from $L^{p_1}(w_1^{p_1})\times L^{p_2}(w_2^{p_2})$ to $L^q(\mu_{\bf w})$, because ${\bf w}\in {\bf A}_{{\bf P},q}$ by Lemma \ref{conditions-weights}.\\

 To show that \eqref{2compact} holds, choose $r$ large so that $\supp b \subset B_r(0)$, then for $|x|>R\geq\max\{2r,1\}$, we have
 \begin{align*}
    |[b,&I_\alpha^\delta](f,g)(x)| \lesssim \int_{\supp b}\int_{\R^n}\frac{|b(y)| |f(y)| |g(z)|}{(|x-y|+|x-z|)^{2n-\alpha}}dzdy\\
    &\lesssim\|b\|_\infty\int_{\supp b}|f(y)| \int_{\R^n} \frac{|g(z)|}{(|x|+|x-z|)^{2n-\alpha}}dzdy\\
    &\lesssim\|b\|_\infty \|f\|_{L^{p_1}(w_1^{p_1})} \left(\int_{B_r(0)} w_1^{-p_1'}dy\right)^{1/p_1'} \int_{\R^n} \frac{|g(z)|}{(|x|+|x-z|)^{2n-\alpha}}dz\\
    &\lesssim\frac{\|b\|_\infty} {|x|^{n-\alpha}} \|f\|_{L^{p_1}(w_1^{p_1})} \left(\int_{B_r(0)}w_1^{-p_1'}dy\right)^{1/p_1'}
    \int_{\R^n} \frac{|g(z)|}{(|x|+|x-z|)^{n}}dz\\
    &\lesssim\frac{\|b\|_\infty}{|x|^{n-\alpha}}\|f\|_{L^{p_1}(w_1^{p_1})} \left( \int_{B_r(0)}w_1^{-p_1'}dy\right)^{1/p_1'}
    \int_{\R^n}\frac{|g(z)|}{(1+|z|)^{n}}dz\\\
    &\lesssim\frac{\|b\|_\infty}{|x|^{n-\alpha}}\|f\|_{L^{p_1}(w_1^{p_1})}\|g\|_{L^{p_1}(w_2^{p_2})}\left(\int_{B_r(0)}w_1^{-p_1'}dy\right)^{1/p_1'}
    \int_{\R^n}\frac{w_2^{-p_2'}}{(1+|z|)^{np_2'}}dz.
    \end{align*}
 Note now, that since $w_2^{p_2q/p}\in A_p\subset A_{p_2}$, we have that $w_2^{-\frac qpp_2'}=w_2^{(p_2q/p)(1-p_2')}$ is in $A_{p_2'}$, and since $q/p>1$, we have that $w_2^{-p_2'}\in A_{p_2'}$ as well. This gives us that $$\int_{\R^n}\frac{w_2^{-p_2'}}{(1+|z|)^{np_2'}}dz<\infty,$$ and so
    $$|[b,I_\alpha^\delta](f,g)(x)|\lesssim\frac1{|x|^{n-\alpha}}.$$

    Raising both sides of the last inequality to the power $q$ and integrating over $|x|>R$ we have
    $$
    \int_{|x|>R}|[b,I_\alpha^\delta](f,g)(x)|^q \mu_{\bf w}\,dx\lesssim \int_{|x|>R}\frac{\mu_{\bf w}}{|x|^{(n-\alpha)q}}dx=\int_{|x|>R}\frac{\mu_{\bf w}}{|x|^{\frac{n-\alpha}{n-p\alpha}np}}dx.
    $$
    Note now that $\frac{n-\alpha}{n-p\alpha}>1$, and that $\mu_{\bf w}$ is an $A_p$ weight by Lemma \ref{conditions-weights}, so this quantity tends to zero as $R\to\infty$.\\

    To show \eqref{3compact}, notice that by adding and subtracting
        $$\int_{\R^n}\int_{\R^n}b(x+t) K^\delta(x,y,z) f(y)g(z) \, dydz,$$
    we can compute
    \begin{align*}
        [b,&I_\alpha^\delta](f,g)(x+t)-[b,I_\alpha^\delta](f,g)(x)\\
                &=\left(b(x)-b(x+t)\right)\int_{\R^n}\int_{\R^n} K^\delta(x,y,z) f(y)g(z) \, dydz \\
        &\,\,\,\, +  \int_{\R^n}\int_{\R^n}\left(b(y)-b(x+t)\right)f(y)g(z)   (K^\delta(x+t,y,z) - K^\delta(x,y,z))\, dydz      \\
                &=I(x,t)+II(x,t).
    \end{align*}
    For $I$, we simply have $$|I(x,t)|\leq|t|\|\nabla b\|_\infty I_\alpha(f,g)(x),$$ and since $I_\alpha$ is bounded from $L^{p_1}(w_1^{p_1})\times L^{p_2}(w_2^{p_2})$ to $L^q(\mu_{\bf w})$, we have
    $$\|I(\cdot,t)\|_{L^q(\mu_{\bf w})}\lesssim|t|.$$

 We now move on to the control of $II$. We can assume $t<\delta/4$. Hence,
 because of the properties of $K^\delta$, if $\max(|x-y|, |x-z|) \leq \delta/2$
  we have
    $$K^\delta(x+t,y,z)-K^\delta(x,y,z)= 0,$$
  while for $\max(|x-y|, |x-z|) > \delta/2$  we have $\max(|x-y|, |x-z|) >2t$. We can then estimate $II$ by
   \begin{align*}
& \left| \iint (b(y)-b(x+t))(K^\delta(x+t,y,z)-K^\delta(x,y,z))f(y)g(z)\,dydz \right| \\
&\lesssim \|b\|_\infty|t|\iint_{\max\{|x-y|,|x-z|\}>\delta/2}\frac{|f(y)||g(z)|}{(|x-y|+|x-z|)^{2n-\alpha+1}}\,dydz\\
&\lesssim \|b\|_\infty|t| \sum_{j\geq 0}\iint_{2^{j-1}\delta < \max\{|x-y|,|x-z|\} \leq 2^{j}\delta}\frac{|f(y)||g(z)|}{(|x-y|+|x-z|)^{2n-\alpha+1}}\,dydz\\
& \lesssim \|b\|_\infty|t| \sum_{j\geq 0} \left( \int_{2^{j-1}\delta \leq|x-z| \leq2 ^{j}\delta} \frac{|f(y)|}{ |x-y|^{2n-\alpha+1} }   \,dy
 \int_ {|x-y| \leq 2 ^{j}\delta} |g(z)|\,dz
\right. \\
&
\,\,\, + \left.  \int_{|x-y| \leq 2 ^{j}\delta}|f(y)| \,dy \int_{ 2 ^{j-1}\delta\leq |x-z| \leq 2 ^{j}\delta} \frac{ |g(z)|}{  |x-z|^{2n-\alpha+1} }   \,dz\right) \\
&
\lesssim   \| b\|_{L^\infty} |t|
\sum_{j\geq 0}  (2^{j}\delta)^{-2n+\alpha-1}\left(
 \int_{|x-y| \lesssim 2^{j} \delta} {|f(y)|}\,dy\,
\int_{|z-y| \lesssim 2^{j}\delta} {|g(z)|}\, dz\right)\\
&
\lesssim   \|b\|_{L^\infty} \, \frac{|t|}{\delta} \,
\sum_{j \geq 0} 2^{-j}  (2^j\delta)^\alpha \left(
 \dashint_{|x-y| \lesssim 2^{j} \delta} {|f(y)|}\,dy\,
\dashint_{|z-y| \lesssim 2^{j}\delta} {|g(z)|}\, dz\right)\\
&
\lesssim   \|b\|_{L^\infty} \, \frac{|t|}{\delta} \,
 \M_\alpha(f,g)(x).
\end{align*}
It follows that
$$\|II(\cdot,t)\|_{L^q(\mu_{\bf w})}\lesssim|t|.$$

Obviously (ii) implies (iii). So it remains to show that (iii) implies (i). To do so we will adapt some arguments from \cite{CDW},
which in turn are based on the original work in \cite{U}. The approach is as follows: we will show that if we assume that $[b,I_\alpha]_1$ is compact and
$b$ (a fortiori in $BMO$ by the the results in \cite{Chaff}) fails to satisfy one of the conditions \eqref{1cmo}-\eqref{3cmo}, then one can construct sequences of functions,
$\{f_j\}_j$ uniformly bounded on $L^{p_1}$ and $\{g_j\}_j$ uniformly bounded on $L^{p_2}$, such that $\{[b,I_\alpha]_1(f_j,g_j)\}_j$
has no convergent subsequence, which contradicts the compactness assumption. It then follows that if $[b,I_\alpha]_1$ is compact,
$b$ must satisfy all three conditions \eqref{1cmo}-\eqref{3cmo} and hence be an element of $CMO$.

Before we construct the sequences, we observe that by linearity in $b$, it is enough to prove that
(iii) implies (i) for $b$ real valued and with $\|b\|_{BMO}=1$. So we will assume such conditions.

Given a cube $Q_j$ such that
  \begin{equation}
  \frac{1}{|Q_j|}\int_{Q_j}|b(x)-b_{Q_j}|dx\geq\eps,\label{awayfromzero}
  \end{equation}
    for some $\eps>0$, we define
    $$f_j(y)=|Q_j|^{-1/p_1}\left({\rm sgn}(b(y)-b_{Q_j})-c_0\right)\chi_{Q_j}(y),$$ where $c_0=|Q_j|^{-1}\int_{Q_j}\text{sgn}(b(y)-b_{Q_j})dy$. Note that $-1<c_0<1$, and from this we see that $f_j$ has the following properties,
    $$\text{supp}f_j\subset Q_j,$$
    $$f_j(y)(b(y)-b_{Q_j})\geq0,$$
    $$\int f_j(y)dy=0,$$
    $$\int (b(y)-b_{Q_j})f_j(y)dy=|Q_j|^{-1/p_1}\int_{Q_j}|b(y)-b_{Q_j}|dy$$
    $$|f_j(y)|\leq 2|Q_j|^{-1/p_1}$$
This last property gives us that $\|f_j\|_{L^{p_1}}\leq2$. For the other functions, we will simply define
 $$g_j=\frac{\chi_{Q_j}}{|Q_j|^{1/p_2}},$$
which satisfies $\|g_j\|_{L^{p_2}}=1$.

Next we establish several technical estimates. For a cube $Q_j$ with center $y_j$ and satisfying \eqref{awayfromzero} for some $\eps>0$, $f_j$ and $g_j$ as above, and all  $x\in\(2\sqrt n Q_j\)^{c}$,
the following point-wise estimates hold:

    \begin{align}
    |I_\alpha((b-b_{Q_j})f_j,g_j)(x)|&\lesssim |Q_j|^{\frac{1}{p'_1}+\frac{1}{p'_2}}|x-y_j|^{-2n+\alpha},  \label{Est1}\\
    |I_\alpha((b-b_{Q_j})f_j,g_j)(x)|&\gtrsim \eps|Q_j|^{\frac{1}{p'_1}+\frac{1}{p'_2}}|x-y_j|^{-2n+\alpha}\label{Est2},\\
    |I_\alpha(f_j,g_j)(x)|&\lesssim |Q_j|^{\frac{1}{p'_1}+\frac{1}{p'_2}+\frac1n}|x-y_j|^{-2n+\alpha-1},\label{Est3}
    \end{align}
where the constants involved are independent of $b, f_j, g_j$ and $\epsilon$.

To prove \eqref{Est1}, we use that   $|x-y_j| \approx |x-y|$ for all $y\in Q_j$ and that $\|b\|_{BMO}=1$ to obtain

\begin{align*}
    |I_\alpha((b-b_{Q_j})&f_j,g_j)(x)|=\left|\int\int\frac{(b(y)-b_{Q_j})f_j(y)g_j(z)}{\(|x-y|^2+|x-z|^2\)^{n-\alpha/2}}dydz\right|\\
       &\lesssim \frac{1}{|Q_j|^{\frac{1}{p_1}+\frac{1}{p_2}}}|x-y_j|^{-2n+\alpha}\int_{Q_j}\int_{Q_j}|b(y)-b_{Q_j}|dydz\\
    &\lesssim |Q_j|^{\frac{1}{p'_1}+\frac{1}{p'_2}}|x-y_j|^{-2n+\alpha}.
    \end{align*}
Using that $(b(y)-b_{Q_j})f_j(y)\geq0$, we can also estimate
 \begin{align*}
    |I_\alpha((b-b_{Q_j})&f_j,g_j)(x)|=\left|\int\int\frac{(b(y)-b_{Q_j})f_j(y)g_j(z)}{\(|x-y|^2+|x-z|^2\)^{n-\alpha/2}}dydz\right|\\
    &\gtrsim |Q_j|^{1-\frac{1}{p_2}}|x-y_j|^{-2n+\alpha}\left|\int_{Q_j}(b(y)-b_{Q_j})f_j(y)dy\right|\\
    &= |Q_j|^{1-\frac{1}{p_2}}|x-y_j|^{-2n+\alpha}\int_{Q_j}(b(y)-b_{Q_j})f_j(y)dy\\
    &= |Q_j|^{1-\frac{1}{p_2}}|x-y_j|^{-2n+\alpha}|Q_j|^{1-\frac{1}{p_1}}\frac1{|Q_j|}\int_{Q_j}|(b(y)-b_{Q_j})|dy\\
    &\geq |Q_j|^{\frac{1}{p'_1}+\frac{1}{p'_2}}|x-y_j|^{-2n+\alpha}\eps,
    \end{align*}
which gives \eqref{Est2}. Finally using that $f_j$ has mean zero we obtain \eqref{Est3} in the following way,
\begin{align*}
    |&I_\alpha(f_j,g_j)(x)|=\left|\int\int\frac{f_j(y)g_j(z)}{\(|x-y|^2+|x-z|^2\)^{n-\alpha/2}}dydz\right|\\
    &=\left|\int\(\int\frac{f_j(y)g_j(z)}{\(|x-y|^2+|x-z|^2\)^{n-\alpha/2}}-
    \frac{f_j(y)g_j(z)}{\(|x-y_j|^2+|x-z|^2\)^{n-\alpha/2}}dy\)dz\right|\\
    &\lesssim \int\int\frac{|y-y_j||f_j(y)|g_j(z)}{\(|x-y_j|+|x-z|\)^{2n-\alpha+1}}dy dz\\
    &\lesssim \frac{|Q_j|^{\frac1n}}{|x-y_j|^{2n-\alpha+1}}\int\int|f_j(y)|g_j(z)dydz\\
        &\lesssim |Q_j|^{\frac1{p'_1}+\frac{1}{p'_2}+\frac1n}|x-y_j|^{-2n+\alpha-1}.
    \end{align*}

Following  \cite{U} and \cite{CDW}, we now use the above point-wise estimates \eqref{Est1}-\eqref{Est3} to prove some $L^q$-norm inequalities
for $[b,I_\alpha]_1(f_j,g_j)$.

For a cube $Q_j$ with center $y_j$, side length $d_j$, and satisfying \eqref{awayfromzero} for some $\eps>0$;  and $f_j$ and $g_j$  defined as above; there exist constants $\gamma_2>\gamma_1>2$, and $\gamma_3>0$, depending only on $p_1,\ p_2,\ n,$ and $\eps$, such that
    \begin{align}
    \(\int_{\gamma_1d_j<|x-y_j|<\gamma_2d_j}|[b,I_\alpha]_1(f_j,g_j)(y)|^qdy\)^{1/q}&\geq\gamma_3\label{C11}\\
    \(\int_{|x-y_j|>\gamma_2d_j}|[b,I_\alpha]_1(f_j,g_j)(y)|^qdy\)^{1/q}&\leq\frac{\gamma_3}4\label{C12}
    \end{align}

Starting with some  $\tilde\gamma_1>16$, using \eqref{Est3} and the fact that $2n-\alpha-n/q>0$ (since $\frac1{p_1}+\frac1{p_2}<2$), we have,
    \begin{align*}
  &  \(\int_{|x-y_j|>\tilde\gamma_1d_j}|(b(x)-b_{Q_j})I_\alpha(f_j,g_j)(x)|^qdx\)^{\frac{1}{q}}\notag\\
    &\leq C|Q_j|^{\frac1{p'_1}+\frac{1}{p'_2}+\frac1n}
    \sum_{s=\lfloor\log_2(\tilde\gamma_1)\rfloor}^\infty\(\int_{2^sd_j<|x-y_j|<2^{s+1}d_j}\frac{|b(x)-b_{Q_j}|^q}{|x-y_j|^{q(2n-\alpha+1)}}\)^{\frac1q}\notag\\
    &\leq C|Q_j|^{\frac1{p'_1}+\frac{1}{p'_2}+\frac1n}\times \\
     &\,\,\,\,\,\,
    \sum_{s=\lfloor\log_2(\tilde\gamma_1)\rfloor}^\infty2^{-s(2n-\alpha+1)}|Q_j|^{-2+\frac\alpha n-\frac1n}\(\int_{2^sd_j<|x-y_j|<2^{s+1}d_j}|b(x)-b_{Q_j}|^q\)^{\frac1q}\notag\\
    &\leq C\sum_{s=\lfloor\log_2(\tilde\gamma_1)\rfloor}^\infty s2^{-s(2n-\alpha-\frac nq+1)}\notag\\
    &\leq C\sum_{s=\lfloor\log_2(\tilde\gamma_1)\rfloor}^\infty 2^{-s(2n-\alpha-\frac nq+\frac12)},
    \end{align*}
 where we have used that for $b\in$ BMO,
    $$\(\int_{2^sd_j<|x-y_j|<2^{s+1}d_j}|b(x)-b_{Q_j}|^qdx\)^{\frac1q}\lesssim s2^{sn/q}|Q_j|^{1/q},$$
    and that $s\leq2^{s/2}$ for $4\leq\lfloor\log_2(\tilde\gamma)\rfloor\leq s$. We thus obtain
  \begin{equation}
    \(\int_{|x-y_j|>\tilde\gamma_1d_j}|(b(x)-b_{Q_j})I_\alpha(f_j,g_j)(x)|^qdx\)^{\frac{1}{q}}
    \leq C \tilde\gamma_1^{-(2n-\alpha-\frac nq+\frac12)}.\label{Est4}
    \end{equation}

    Next, for $\tilde\gamma_2>\tilde\gamma_1$, using \eqref{Est2} and \eqref{Est4}, we obtain the following,
    \begin{align}
    & \(\int_{\tilde\gamma_1d_j<|x-y_j|<\tilde\gamma_2d_j}|[b,I_\alpha]_1(f_j,g_j)(x)|^qdx\)^{\frac1q} \nonumber\\
    &\geq C\(\int_{\tilde\gamma_1d_j<|x-y_j|<\tilde\gamma_2d_j}|I_\alpha\((b-b_Q)f_j,g_j\)(x)|^qdx\)^{\frac1q} \nonumber\\
    &\ \ \ \ -C\(\int_{\tilde\gamma_1d_j<|x-y_j|}|(b(x)-b_Q)I_\alpha(f_j,g_j)(x)|^qdx\)^{\frac1q} \nonumber\\\
    &\geq C\eps|Q_j|^{\frac{1}{p_1'}+\frac{1}{p_2'}}\(\int_{\tilde\gamma_1d_j<|x-y_j|<\tilde\gamma_2d_j}|x-y_j|^{q(-2n+\alpha)}dx\)^{\frac1q} \nonumber\\
    &\ \ \ \ -C\tilde\gamma_1^{(-2n+\alpha+n/q-1/2)}\nonumber\\
    &\geq C\eps\(\tilde\gamma_1^{-2nq+n+\alpha q}-\tilde\gamma_2^{-2nq+n+\alpha q}\)^{\frac1q}-C\tilde\gamma_1^{(-2n+\alpha+n/q-1/2)}. \label{Est4'}
    \end{align}
Using \eqref{Est4} and \eqref{Est4'} we see that we can select $\gamma_1, \gamma_2$ in place of $\tilde \gamma_1, \tilde \gamma_2$, with
$\gamma_2>>\gamma_1$, so that \eqref{C11} and \eqref{C12} are verified for some $\gamma_3>0$.\\

 The final technical estimate we need is the following.      Given $\gamma_1,\gamma_2$ in \eqref{C11} and \eqref{C12},
  there exists a $0<\beta<<\gamma_2$ depending only on $p_1,\ p_2,\ n,$ and $\eps$ such that for any $E$ measurable such that
  $$E\subset\{x:\gamma_1d_j<|x-y_j|<\gamma_2d_j\}$$
   and $|E|/|Q_j|<\beta^n$,
   we have
    \begin{align}
    \(\int_E|[b,I_\alpha]_1(f_j,g_j)(y)|^qdy\)^{1/q}&\leq\frac{\gamma_3}4.\label{C2}
    \end{align}

To prove this last inequality we note that if  $E\subset\{x:\gamma_1d_j<|x-y_j|<\gamma_2d_j\}$ is measurable, we can use \eqref{Est1} and \eqref{Est3} to get,
    \begin{align}
    \(\int_E|[b,I_\alpha]_1(f_j,g_j)(x)|^qdx\)^{\frac1q}&\lesssim |Q_j|^{\frac{1}{p_1'}+\frac{1}{p_2'}}\(\int_E|x-y_j|^{-q(2n-\alpha)}dx\)^\frac{1}{q}\notag\\
    &\ \ \ \ +|Q_j|^{\frac{1}{p_1'}+\frac1{p_2'}+\frac1n}\(\int_E\frac{|b(x)-b_{Q_j}|}{|x-y_j|^{q(2n-\alpha+1)}}dx\)^{\frac1q}\notag\\
    &\lesssim \(\frac{|E|^{1/q}}{|Q_j|^{1/q}}+\(\frac{1}{|Q_j|}\int_{E}|b(x)-b_{Q_j}|^qdx\)^{\frac1q}\)\label{star}
    \end{align}
From here the arguments in \cite{CDW} can be followed identically, and it shown there that there exists some positive constant $\tilde C$ depending on $\gamma_1,\ \gamma_2,$ and $b$ such that
$$\eqref{star}\lesssim \frac{|E|^{1/q}}{|Q_j|^{1/q}}\(1+\log\(\frac{\tilde C|Q_j|}{|E|}\)\)^{\frac{\lfloor q\rfloor +1}q}$$
(see \cite[p.309]{CDW}). Clearly we can now select $0< \beta < \min(\tilde C^{1/n},\gamma_2)$ and sufficiently small so that \eqref{C2} holds.

 We are left with constructing the sequences that will lead to a contradiction depending on which of the conditions \eqref{1cmo}-\eqref{3cmo} $b$ is supposed to fail to satisfy. The arguments are again borrowed from \cite{CDW} but adapted to our bilinear situation.

 If $b$ does not satisfy \eqref{1cmo}, then there exists some $\eps>0$ and a sequence $\{Q_j\}$ with $|Q_j|\to0$ as $j\to\infty$ such that for every $j$,
        \begin{align}
        \eps\leq\frac{1}{|Q_j|}\int_{Q_j}|b(y)-b_{Q_j}|dy\label{eps}
        \end{align}
        We then can pick a subsequence, which we will denote $\{Q_j^{(i)}\}$, so that
        \begin{align*}
        \frac{d_{j+1}^{(i)}}{d_{j}^{(i)}}&<\frac\beta{2\gamma_2}.
        \end{align*}
We also let $f_j^{(i)}$ and $g_j^{(i)}$ be the sequences associated to the selected cubes $Q_j^{(i)}$ as defined earlier on.

 For fixed $k$ and $m$, we define the following sets,
\begin{align*}
  G&=\{x:\gamma_1d^{(i)}_k<|x-y_k^{(i)}|<\gamma_2d_k^{(i)}\},\\
  G_1&=G\setminus\{x:|x-y_{k+m}^{(i)}|\leq\gamma_2d_{k+m}^{(i)}\},\\
  G_2&=\{x:|x-y_{k+m}^{(i)}|>\gamma_2d_{k+m}^{(i)}\}.
\end{align*}
Note that since $G_1=G\cap G_2$, we have,
\begin{align}
G_1&\subset G_2\label{G1}\\
G_1&=G\setminus\(G_2^c\cap G\)\label{G2}.
\end{align}
Also, by construction and our choice of $Q_j^{(i)}$'s, one can easily see that
\begin{align}
\frac{|G_2^c\cap G|}{|Q_k^{(i)}|}\leq\beta^n\label{G3},
\end{align}
see \cite[p.307]{CDW}.
It follows that
\begin{align*}
  \|[b,I_\alpha]_1&(f_k^{(i)},g_k^{(i)})-[b,I_\alpha]_1(f_{k+m}^{(i)},g_{k+m}^{(i)})\|_{L^q}\\
  &\geq\(\int_{G_1}|[b,I_\alpha]_1(f_k^{(i)},g_k^{(i)})-[b,I_\alpha]_1(f_{k+m}^{(i)},g_{k+m}^{(i)})|^q\)^{\frac1q}\\
  &\geq\(\int_{G_1}|[b,I_\alpha]_1(f_k^{(i)},g_k^{(i)})|^q\)^{\frac1q}-\(\int_{G_1}|[b,I_\alpha]_1(f_{k+m}^{(i)},g_{k+m}^{(i)})|^q\)^{\frac1q}\\
  &\geq\(\int_{G_1}|[b,I_\alpha]_1(f_k^{(i)},g_k^{(i)})|^q\)^{\frac1q}-\(\int_{G_2}|[b,I_\alpha]_1(f_{k+m}^{(i)},g_{k+m}^{(i)})|^q\)^{\frac1q}\\
  &=\(\int_{G}|[b,I_\alpha]_1(f_k^{(i)},g_k^{(i)})|^q - \int_{G_2^c\cap G}|[b,I_\alpha]_1(f_k^{(i)},g_k^{(i)})|^q\)^{\frac1q} \\
  &\,\,\,\,\,\,\,\,\,\,\,\,\,\,-\(\int_{G_2}|[b,I_\alpha]_1(f_{k+m}^{(i)},g_{k+m}^{(i)})|^q\)^{\frac1q}.
 \end{align*}
Using \eqref{C11}, \eqref{C2}, and \eqref{C12} in each of the three terms above we finally arrive at
\begin{align*}
  \|[b,I_\alpha]_1(f_k^{(i)},g_k^{(i)})-[b,I_\alpha]_1(f_{k+m}^{(i)},g_{k+m}^{(i)})\|_{L^q}
&\geq\(\gamma_3^q-\frac{\gamma_{3}^q}{4^q}\)^{\frac1q}-\frac{\gamma_3}{4}\\
  &\gtrsim\frac{\gamma_3}2.
   \end{align*}
Since every pair of terms in the sequence $\{[b,I_\alpha]_1(f_j^{(i)},g_j^{(i)})\}$ are at least $C\gamma_3$ apart from each other, there can be no convergent subsequence, and therefore $[b,I_\alpha]_1$ would not be compact.  So $b$ must satisfy \eqref{1cmo}.\\

If $b$ violates \eqref{2cmo}, we again have that there exists $\eps$ and sequence of cubes $\{Q_j\}$, this time with $|Q_j|\to\infty$ as $j\to\infty$, such that \eqref{eps} is satisfied. This time we take the subsequence $\{Q_j^{(ii)}\}$ so that $$\frac{d_{j}^{(ii)}}{d_{j+1}^{(ii)}}<\frac\beta{2\gamma_2}.$$
We can use a similar method as in the previous case, but since our diameters are increasing instead of decreasing, we simply define our sets in a `reversed' order, so for fixed $k$ and $m$, we have
\begin{align*}
  G&=\{x:\gamma_1d^{(ii)}_{k+m}<|x-y_{k+m}^{(ii)}|<\gamma_2d_{k+m}^{(ii)}\},\\
  G_1&=G\setminus\{x:|x-y_{k}^{(ii)}|\leq\gamma_2d_{k}^{(ii)}\},\\
  G_2&=\{x:|x-y_{k}^{(ii)}|>\gamma_2d_{k}^{(ii)}\}.
\end{align*}
As before we have that \eqref{G1}-\eqref{G3} hold, and so from here, the calculations are identical to those in the first case.\\

Finally, if \eqref{3cmo} is not satisfied, there exists some cube $Q$ with diameter $d$, and some sequence $\{y_j\}$, with $|y_j|\to\infty$, such that \eqref{eps} holds for $\{Q_j:=Q+y_j\}$. We then let $B_j=\{x\in\nR^n:|x-y_j|<\gamma_2d\}$, and choose $\{Q_j^{(iii)}\}$ so that $B_{j}\cap B_{k}=\emptyset$ for $j\neq k$.\\

Note that by the construction of the balls $B_j$, if we define $G,\ G_1,$ and $G_2$ as in $(i)$, we in fact have that $G=G_1=G\cap G_2$, and so $G_2^c\cap G=\emptyset$. This means that while the calculations for this case could certainly be simplified, it is sufficient to once again repeat the steps preformed in the first case to obtain the desired result.
\end{proof}

\end{document}